\documentclass[a4paper,UKenglish]{llncs}
\usepackage{amsmath}
\usepackage{amsfonts}
\usepackage{amssymb}
\usepackage{hyperref}
\usepackage{graphics}
\usepackage{graphicx}
\usepackage{pstricks}
\usepackage{pst-node}
\usepackage{stmaryrd}
\usepackage{float}
\usepackage{bussproofs}
\usepackage{microtype}%if unwanted, comment out or use option "draft"
\usepackage{color}
\newcommand{\Rcal}{{\cal R}}
\newcommand{\seq}{\vdash}
\begin{document}

\title{Schematic Cut elimination and the Ordered Pigeonhole Principle \subtitle{[Extended Version]} }

\author{David Cerna \inst{1} \and Alexander Leitsch \inst{2}}
\institute{Research Institute for Symbolic
Computation (RISC) \\ Johannes Kepler University, Linz, Austria \\ \href{mailto:dcerna@risc.uni-linz.ac.at}{dcerna@risc.uni-linz.ac.at}
\and
Logic and Theory Group\\ Technical University of Vienna \\  \href{mailto: leitsch@logic.at}{ leitsch@logic.at} }

\authorrunning{D.\,M. Cerna & A. Leitsch} %mandatory. First: Use abbreviated first/middle names. Second (only in severe cases): Use first author plus 'et. al.'

\maketitle

\begin{abstract}
In previous work, an attempt was made to apply the {\em schematic CERES method} \cite{CERESS2} to a formal proof with an arbitrary number of $\Pi_{2}$ cuts (a recursive proof encapsulating the infinitary pigeonhole principle) \cite{MeCadePaper2015}. However the derived schematic refutation for the {\em characteristic clause set} of the proof  could not be expressed in the formal language provided in \cite{CERESS2}.  Without this formalization a {\em Herbrand system} cannot be algorithmically extracted. In this work, we provide a restriction of the proof found in \cite{MeCadePaper2015}, the {\em ECA-schema (Eventually Constant Assertion), or ordered infinitary pigeonhole principle}, whose analysis can be completely carried out in the framework of \cite{CERESS2}, this is the first time the framework is used for proof analysis. From the refutation of the clause set and a substitution schema we construct a {\em Herbrand system}. 

%The differences between the ECA-schema and the proof of \cite{MeCadePaper2015} elude to the issues of schematic proof analysis in the framework of \cite{CERESS2}. 
 \end{abstract}

\section{Introduction}

For his famous {\em Hauptsatz}~\cite{Gentzen1935}, Gerhard Gentzen developed the sequent calculus \textbf{LK}. Gentzen went on to show that the {\em cut} inference rule is redundant and in doing so, was able to show several results on consistency and decidability. The method he developed for eliminating cuts from \textbf{LK}-derivations works by inductively reducing the cuts in a given \textbf{LK}-derivation to cuts which either have a reduced {\em formula complexity} and/or reduced {\em  rank}~\cite{prooftheory}. This method of cut elimination is known as {\em reductive cut elimination}. A useful consequence of cut elimination for the \textbf{LK}-calculus is that cut-free \textbf{LK}-derivations have the {\em subformula property},  i.e. every formula occurring in the derivation is a subformula of some formula in the end sequent. This property admits the construction of {\em Herbrand sequents} and other objects which are essential in proof analysis. 

By using the technique of cut-elimination, it is also possible to gain mathematical knowledge concerning the connection between different proofs of the same theorem. For example, Jean-Yves Girard's application of cut elimination to the F\"{u}rstenberg-Weiss' proof of van der Waerden's theorem~\cite{ProocomWaerdens1987} resulted in the {\em analytic} proof of van der Waerden's theorem as found by van der Waerden himself. From the work of Girard, it is apparent that interesting results can be derived from eliminating cuts in ``mathematical'' proofs. 

A more recently developed method of cut elimination, the CERES method~\cite{CERES}, provides the theoretic framework to directly study the cut structure of \textbf{LK}-derivations, and in the process reduces the computational complexity of deriving a cut-free proof. The cut structure is transformed into a clause set allowing for a clausal analysis of the resulting clause form. Methods of reducing clause set complexity, such as {\em subsumption} and {\em tautology elimination} can be applied to the characteristic clause set to increase the efficiency. It was shown by Baaz \& Leitsch in ``Methods of cut Elimination''~\cite{Baaz:2013:MC:2509679} that this method of cut elimination has a {\em non-elementary speed up} over reductive cut elimination.

In the same spirit of Girard's work, the CERES method was applied to a formalization of  F\"{u}rstenberg's proof of the infinitude of primes~\cite{Baaz:2008:CAF:1401273.1401552}. Instead of formalizing the proof as a single proof (in second-order arithmetic) it was represented as a sequence of first-order proofs enumerated by a single numeric parameter indexing the number of primes assumed to exist (leading to a contradiction). The resulting schema of clause sets was refuted for the first few instances by the system CERES. The general refutation schema, resulting in Euclid's construction of primes, was specified on the mathematical meta-level.  At that time no object-level construction of the resolution refutation schema existed. 

A straightforward mathematical formalization of  F\"{u}rstenberg's proof requires induction. In higher-order logic, induction is easily formalized via the comprehension principle. However in first-order, an induction rule or induction axioms have to be added to the \textbf{LK}-calculus. As it was shown in~\cite{CERESS2}, ordinary reductive cut elimination does not work in the presence of an induction rule in the \textbf{LK}-calculus. There are, however, other systems~\cite{Mcdowell97cut-eliminationfor} which provide cut-elimination in the presence of an induction rule; but these systems do not produce proofs with the subformula property, which is necessary for Herbrand system extraction.   

In ``Cut-Elimination and Proof Schemata''~\cite{CERESS2}, a version of the \textbf{LK}-calculus was introduced (\textbf{LKS}-calculus) allowing for the formalization of sequences of proofs as a single object level construction, i.e. {\em proof schema}, as well as a framework for performing a CERES-type cut elimination on proof schemata. Cut elimination performed within the framework of~\cite{CERESS2} results in cut-free proof schemata with the subformula property. 

In previous work, we applied the schematic CERES method of~\cite{CERESS2} to a proof formalized in the \textbf{LKS}-calculus\cite{MeCadePaper2015,MyThesis}. We referred to this formal proof as the {\em Non-injectivity  Assertion} (NiA) schema. A well known variation of the NiA-schema, of which has been heavily studied in literature, is the {\em infinitary Pigeonhole Principle} (PHP). Though a resolution refutation schema was found and mathematically specified~\cite{MeCadePaper2015}, it was not possible to express this refutation schema within the language of~\cite{CERESS2}.  The main problem  was the specification of a unification and refutation schema. This issue points to a fundamental property of CERES-based schematic cut-elimination, namely that the language for specifying the refutation schema is more complex than that specifying the proof schema.

In this work we construct a formal proof for a weaker variant of the NiA-schema which we call the {\em Eventually Constant Assertion schema} (ECA-schema). The ECA is an encapsulation of the infinitary pigeonhole principle where the holes are ordered. For the ECA-schema a specification of the resolution refutation schema within the formalism of~\cite{CERESS2} turned out successful. In particular, we are able to extract a Herbrand system and complete the proof analysis of the ECA-schema.

The paper is structured as follows: In Sec. \ref{sec:SCERES}, we introduce the \textbf{LKS}-calculus and the essential concepts from~\cite{CERESS2} concerning the schematic clause set analysis. In Sec. \ref{sec:MathNiA}, we mathematically prove the  ECA-schema. We leave the formal proof, written in the \textbf{LKS}-calculus, to Appendix  \ref{sec:FormECA}. In Sec. \ref{sec:CCSSE}, we extract the characteristic clause set from the ECA-schema and perform {\em normalization} and tautology elimination. In Sec. \ref{sec:refuteset}, we provide a refutation of the extracted characteristic clause set. In Sec. \ref{sec:Herbrand}, we extract a Herbrand system for the refutation of  Sec. \ref{sec:refuteset}. In Sec. \ref{sec:Conclusion}, we conclude the paper and discuss our conjecture.

\section{The \textbf{LKS}-calculus and Clause set Schema}\label{sec:SCERES}

In this section we introduce the \textbf{LKS}-calculus, which will be used to formalize the ECA-schema, and the schematic CERES method. 

\subsection{Schematic language, proofs, and the \textbf{LKS}-calculus}
The \textbf{LKS}-calculus is based on the \textbf{LK}-calculus constructed by Gentzen~\cite{Gentzen1935}. When one grounds the {\em parameter} indexing an \textbf{LKS}-derivation, the result is an  \textbf{LK}-derivation~\cite{CERESS2}. The term language used is extended to accommodate  the schematic constructs of  \textbf{LKS}-derivations. We work in a two-sorted setting containing a {\em schematic sort} $\omega$ and an {\em individual sort} $\iota$. The schematic sort contains numerals constructed from the constant $0:\omega$, a monadic function $s(\cdot):\omega \rightarrow \omega$ as well as $\omega$-variables $\mathcal{N}_{v}$, of which one variable, the {\em free parameter}, will be used to index \textbf{LKS}-derivations. When it is not clear from context, we will represent numerals as $\overline{m}$. The free parameter will be represented by $n$ unless otherwise noted. 

The individual sort is constructed in a similar fashion to the standard first order language~\cite{prooftheory} with the addition of schematic functions. Thus,  $\iota$ contains countably many constant symbols, countably many {\em constant function symbols}, and  {\em defined function symbols}. The constant function symbols are part of the  standard first order language and the defined function symbols are used for schematic terms. Though, defined function symbols can also unroll to numerals and thus can be of type $\omega^n \to \omega$.  The $\iota$ sort also has {\em free} and {\em bound} variables and an additional concept, {\em extra variables}~\cite{CERESS2}. These are variables introduced during the unrolling of defined function ({\em predicate}) symbols. We do not use extra variables in the formalization of the ECA-schema. Also important are the {\em schematic variable symbols} which are variables of type $\omega \rightarrow \iota$. Essentially second order variables, though, when evaluated with a {\em ground term} from the $\omega$ sort we treat them as first order variables. Our terms are built inductively using constants and variables as a base.  

Formulae are constructed inductively using countably many {\em predicate constants}, logical operators $\vee$,$\wedge$,$\rightarrow$,$\neg$,$\forall$, and $\exists$, as well as  {\em defined predicate symbols} which are used to construct schematic formulae. In this work {\em iterated $\bigvee$} is the only defined predicate symbol used. Its formal specification is:
\begin{equation}
\label{eq:one}
\varepsilon_{\vee}= \bigvee_{i=0}^{s(y)} P(i) \equiv \left\lbrace \begin{array}{c}
{\displaystyle \bigvee_{i=0}^{s(y)} P(i) \Rightarrow \bigvee_{i=0}^{y} P(i) \vee P(s(y)) }\\
{\displaystyle \bigvee_{i=0}^{0} P(i) \Rightarrow P(0)}
\end{array}\right. 
\end{equation}
 From the above described term and formulae language we can provide the inference rules of the \textbf{LKE}-calculus, essentially the \textbf{LK}-calculus~\cite{prooftheory} plus an equational theory $\varepsilon$ (in our case $\varepsilon_{\vee}$ Eq. \ref{eq:one}). This theory, concerning our particular usage, is a primitive recursive term algebra describing the structure of the defined function (predicate) symbols. The \textbf{LKE}-calculus is the base calculus for the \textbf{LKS}-calculus which also includes {\em proof links}.
\begin{definition}[$\varepsilon$-inference rule]

\begin{prooftree}
\AxiomC{$S\left[ t\right] $}
\RightLabel{$(\varepsilon)$}
\UnaryInfC{$S\left[ t'\right] $}
\end{prooftree}
In the $\varepsilon$ inference rule, the term $t$ in the sequent $S$ is replaced by a term $t'$ such that, given the equational theory  $\varepsilon$,  $\varepsilon \models t = t'$.
\end{definition}

To extend the \textbf{LKE}-calculus with  proof links we need a countably infinite set of {\em proof symbols}  denoted by $\varphi, \psi,\varphi_{i}, \psi_{j} \ldots$. Let $S(\bar{x})$ by a  sequent with a vector of schematic variables $\bar{x}$, by  $S(\bar{t})$ we denote the sequent $S(\bar{x})$ where each of the variables in $\bar{x}$ is replaced by the terms in the vector $\bar{t}$ respectively, assuming that they have the appropriate type. Let $\varphi$ be a proof symbol and $S(\bar{x})$ a sequent, then the expression \AxiomC{$(\varphi(\bar{t}))$}
\dashedLine
\UnaryInfC{$S(\bar{t})$}
\DisplayProof
is called a {\em proof link} . For a variable $n:\omega$, proof links
such that the only $\omega$-variable is $n$ are called {\em $n$-proof links} \index{k-proof Link}.

\begin{definition}[\textbf{LKE}-calculus~\cite{CERESS2}]
The sequent calculus $\mathbf{LKS}$
consists of the rules of $\mathbf{LKE}$, where proof links may appear
at the leaves of a proof.
\end{definition}

\begin{definition}[Proof schemata~\cite{CERESS2}]\label{def.proofschema}
\index{Proof Schemata}
  Let $\psi$ be a proof symbol and $S(n,\bar{x})$ be a sequent
  such that $n:\omega$. Then a {\em proof schema pair for $\psi$} is a pair of $\mathbf{LKS}$-proofs $(\pi,\nu(k))$ with end-sequents $S(0,\bar{x})$ and $S(k+1,\bar{x})$ respectively such that $\pi$ may not contain proof links and $\nu(k)$ may
  contain only proof links of the form \AxiomC{$(\psi(k,\bar{a}))$}
  \dashedLine
  \UnaryInfC{$S(k,\bar{a})$}
  \DisplayProof 
and we say that it is a proof link to $\psi$. We call $S(n,\bar{x})$ the end sequent of $\psi$ and assume an identification between the formula occurrences in the end sequents of $\pi$ and $\nu(k)$ so that we can speak of occurrences in the end sequent of $\psi$. Finally a proof schema $\Psi$ is a tuple of proof schema pairs for $\psi_1 , \cdots \psi_\alpha$ written as $\left\langle \psi_1 , \cdots \psi_\alpha \right\rangle$, such that the $\mathbf{LKS}$-proofs for $\psi_{\beta}$ may also contain $n$-proof links to $\psi_{\gamma}$ for $1\leq \beta < \gamma\leq \alpha$. We also say that the end sequent of $\psi_1$ is the end sequent of $\Psi$. 
\end{definition}

We will not delve further into the structure of proof schemata and instead refer the reader to~\cite{CERESS2}. We now introduce the {\em characteristic clause set schema}.

\subsection{Characteristic Clause set Schema}
The construction of the characteristic clause set as described for the CERES method~\cite{CERES} required inductively following the formula occurrences of cut formula ancestors up the proof tree to the leaves. However, in the case of proof schemata, the concept of ancestors and formula occurrence is more complex. A formula occurrence might be an ancestor of a cut formula in one recursive call and in another it might not. Additional machinery is necessary to extract the characteristic clause term from proof schemata. A set $\Omega$ of formula occurrences from the end-sequent of an \textbf{LKS}-proof $\pi$ is called {\em a configuration for $\pi$}. A configuration $\Omega$ for $\pi$ is called relevant w.r.t. a proof schema $\Psi$ if $\pi$ is a proof in $\Psi$ and there is a $\gamma \in \mathbb{N}$ such that $\pi$ induces a subproof $\pi\downarrow \gamma$ of $\Psi \downarrow \gamma$
such that the occurrences in $\Omega$ correspond to cut-ancestors below $\pi\downarrow \gamma$~\cite{thesis2012Tsvetan}. Note that the set of relevant cut-configurations can be computed given a proof schema $\Psi$. To represent a proof symbol $\varphi$ and configuration $\Omega$ pairing in a clause set we assign them a {\em clause set symbol} $cl^{\varphi,\Omega}(a,\bar{x})$, where $a$ is a term of the $\omega$ sort. 

\begin{definition}[Characteristic clause term~\cite{CERESS2}]\label{def:charterm}
\index{Characteristic Term}
Let $\pi$ be an $\mathbf{LKS}$-proof and $\Omega$ a configuration. In the following, by $\Gamma_{\Omega}$ , $\Delta_{\Omega}$ and $\Gamma_{C}$ , $\Delta_{C}$ we will denote multisets of formulas of $\Omega$- and $cut$-ancestors respectively. Let $r$ be an inference in $\pi$. We define the clause-set term $\Theta_r^{\pi,\Omega}$ inductively:
\begin{itemize}
\item if $r$ is an axiom of the form $\Gamma_{\Omega} ,\Gamma_C , \Gamma \vdash \Delta_{\Omega} ,\Delta_C , \Delta$, then \\ $\Theta_{r}^{\pi,\Omega} = \left\lbrace \Gamma_{\Omega} ,\Gamma_C  \vdash \Delta_{\Omega} ,\Delta_C \right\rbrace $
\item if $r$ is a proof link of the form
\AxiomC{$\psi(a,\bar{u})$}
\dashedLine
\UnaryInfC{$\Gamma_{\Omega} ,\Gamma_C , \Gamma \vdash \Delta_{\Omega} ,\Delta_C , \Delta$}
\DisplayProof
then define $\Omega'$ as the set of formula occurrences from $\Gamma_{\Omega} ,\Gamma_C  \vdash \Delta_{\Omega} ,\Delta_C$ and $\Theta_{r}^{\pi,\Omega} = cl^{\psi,\Omega}(a,\bar{u})$
\item if $r$ is a unary rule with immediate predecessor \index{Predecessor} $r'$ , then $\Theta_{r}^{\pi,\Omega} =  \Theta_{r'}^{\pi,\Omega}$

\item if $r$ is a binary rule with immediate predecessors $r_1 $, $r_2 $, then 
\begin{itemize}
\item if the auxiliary formulas of $r$ are $\Omega$- or $cut$-ancestors, then
$\Theta_{r}^{\pi,\Omega} = \Theta_{r_1}^{\pi,\Omega} \oplus \Theta_{r_2}^{\pi,\Omega}$
\item otherwise, $\Theta_{r}^{\pi,\Omega} = \Theta_{r_1}^{\pi,\Omega} \otimes \Theta_{r_2}^{\pi,\Omega}$
\end{itemize}
\end{itemize}
Finally, define $\Theta^{\pi,\Omega} = \Theta_{r_0}^{\pi,\Omega}$ where $r_0$ is the last inference in $\pi$ and $\Theta^{\pi} = \Theta^{\pi,\emptyset}$. We call $\Theta^{\pi}$ the characteristic term of $\pi$. 
\end{definition}

Clause terms evaluate to sets of clauses by $|\Theta| = \Theta$ for clause sets $\Theta$, $|\Theta_1 \oplus \Theta_2| = |\Theta_1| \cup |\Theta_2|$, $|\Theta_1 \otimes \Theta_2| = \{C \circ D \mid C \in |\Theta_1|, D \in |\Theta_2|\}$.

The characteristic clause term is extracted for each proof symbol in a given proof schema $\Psi$, and together they make the {\em characteristic clause set schema} for $\Psi$, $CL(\Psi)$.
\begin{definition}[Characteristic Term Schema\cite{CERESS2}]
Let $\Psi = \left\langle \psi_{1},\cdots, \psi_{\alpha} \right\rangle $ be a proof schema. We define the rewrite rules for clause-set symbols for all proof symbols $\psi_{\beta}$  and configurations $\Omega$ as $cl^{\psi_{\beta},\Omega}(0,\overline{u}) \rightarrow \Theta^{\pi_{\beta},\Omega}$ and $cl^{\psi_{\beta},\Omega}(k+1,\overline{u}) \rightarrow \Theta^{\nu_{\beta},\Omega}$ where $1\leq \beta\leq \alpha$. Next, let $\gamma\in \mathbb{N}$ and $cl^{\psi_{\beta},\Omega}\downarrow_{\gamma}$ be the normal form of $cl^{\psi_{\beta},\Omega}(\gamma,\overline{u})$ under the rewrite system just given extended by rewrite rules for defined function and predicate symbols. Then define $\Theta^{\psi_{\beta},\Omega} = cl^{\psi_{\beta},\Omega}$ and $\Theta^{\Psi,\Omega} = cl^{\psi_{1},\Omega}$ and finally the characteristic term schema $\Theta^{\Psi} =  \Theta^{\Psi,\emptyset}$.
\end{definition}
\subsection{Resolution Proof Schemata} 
From the characteristic clause set we can construct {\em clause schemata} which are an essential part of the definition of {\em resolution terms} and {\em resolution proof schema }\cite{CERESS2}. Clause schemata serve as the base for the resolution terms used to construct a resolution proof schema. One additional notion needed for defining resolution proof schema is that of {\em clause variables}. The idea behind clause variables is that parts of the clauses at the leaves can be passed down a refutation to be used later on. The definition of resolution proof schemata uses clause variables as a way to handle this passage of clauses. Substitutions on clause variables are defined in the usual way. 

\begin{definition}[Clause Schema \cite{CERESS2}]
Let $b$ be an numeric term, $\overline{u}$ a vector
of schematic variables and $\overline{X}$  a vector of clause variables. Then $c(b, \overline{u}, \overline{X})$ is
a clause schema w.r.t. the rewrite system $R$:
\begin{center}
$c(0, \overline{u}, \overline{X}) \rightarrow C \circ X$
and
$c(k + 1, \overline{u}, \overline{X}) \rightarrow c(k, \overline{u}, \overline{X}) \circ D$
\end{center}
where $C$ is a clause with $V(C) \subseteq \left\lbrace \overline{u} \right\rbrace$  and $D$ is a clause with $V(D) \subseteq \left\lbrace k, \overline{u}\right\rbrace $. Clauses and clause variables are clause schemata w.r.t. the empty rewrite system.
\end{definition}
\begin{definition}[Resolution Term \cite{CERESS2}]
Clause schemata are resolution terms; if
$\rho_1$ and $\rho_2$  are resolution terms, then $r(\rho_1 ; \rho_2 ; P)$ is a resolution term, where $P$ is an atom formula schema.
\end{definition}

The idea behind the resolution terms is that in the term  $r(\rho_1 ; \rho_2 ; P)$, $P$ is the resolved atom of the resolvents $\rho_1, \rho_2$. The notion of most general unifier has not yet been introduced being that we introduce the concept as a separate schema from the resolution proof schema. 

\begin{definition}[Resolution Proof Schema \cite{CERESS2}]
A resolution proof schema $\mathcal{R}(n)$ is a structure $( \varrho_1 , \cdots , \varrho_\alpha )$ together with a set of rewrite rules $\mathcal{R} = \mathcal{R}_1 \cup \cdots \cup \mathcal{R}_{\alpha}$ ,
where the $\mathcal{R}_i\ (for \ 1 \leq i \leq \alpha )$ are pairs of rewrite rules 
\begin{center}
$\varrho_i (0, \overline{w},\overline{ u},  \overline{X} ) \rightarrow \eta_i$\\
and\\
$\varrho_i (k+1,\overline{w},\overline{ u},  \overline{X} ) \rightarrow \eta'_i $
\end{center}

where, $\overline{w},\overline{ u},$ and $ \overline{X}$ are vectors of $\omega$, schematic, and clause variables respectively, $\eta_i$ is a resolution term over terms of the form $\varrho_j(a_j , \overline{m},\overline{ t},  \overline{C})$ for $i<j\leq \alpha$, and $\eta'_i$ is a resolution term over terms of the form $\varrho_j(a_j , \overline{m},\overline{ t},  \overline{C})$ and $\varrho_i(k, \overline{m},\overline{ t},  \overline{C})$ for $i < j \leq \alpha$; by $a_j$, we denote a term of the  $\omega$ sort.
\end{definition}

The idea behind the definition of resolution proof schema is that the definition simulates a recursive construction of a resolution derivation tree and can be unfolded into a tree once the free parameter is instantiated. The expected properties of resolution and resolution derivations hold for resolution proof schema, more detail can be found in \cite{CERESS2}. 

\begin{definition}[Substitution Schema \cite{CERESS2}]
Let $u_1 , \cdots , u_{\alpha}$ be schematic variable
symbols of type $\omega \rightarrow \iota$ and $t_1 , \cdots , t_{\alpha}$ be term schemata containing no other $\omega$-variables than $k$. Then a substitution schema is an expression of the form $\left[ u_1 /\lambda k.t_{1} , \cdots , u_{\alpha} /\lambda k.t_{\alpha} \right]$.
\end{definition}

Semantically, the meaning of the substitution schema is for all $\gamma\in \mathbb{N}$ we have a substitution of the form $\left[ u_1(\gamma) /\lambda k.t_{1}\downarrow_{\gamma} , \cdots , u_{\alpha}(\gamma) /\lambda k.t_{\alpha}\downarrow_{\gamma} \right]$. For the resolution proof schema the semantic meaning is as follows, Let $R(n) = ( \varrho_{1}, \cdots , \varrho_{\alpha} )$ be a resolution proof schema, $\theta$ be a clause substitution, $\nu$ an $\omega$-variable substitution, $\vartheta$ be a substitution schema, and $\gamma \in \mathbb{N}$, then  $R(\gamma)\downarrow$ denotes a resolution
term which has a normal form of $\varrho_1 (n,\overline{w}, \overline{u} , \overline{X} )\theta\nu\vartheta[n/\gamma ]$ w.r.t. $R$ extended by rewrite rules for defined function and predicate symbols.

\subsection{Herbrand Systems}

From the resolution proof schema and the substitution schema we can exact a so-called {\em Herbrand system}. The idea is to generalize the mid sequent theorem of Gentzen to proof schemata \cite{Baaz:2013:MC:2509679,prooftheory}. This theorem states that a proof (cut-free or with quantifier-free cuts) of a prenex end-sequent can be transformed in a way that there is a midsequent separating quantifier inferences from propositional ones. The mid-sequent is propositionally valid (w.r.t. the axioms) and contains (in general several) instances of the matrices of the prenex formulas; it is also called a {\em Herbrand sequent}. The aim of this paper is to extract schematic Herbrand sequents from schematic cut-elimination via CERES. We restrict the sequents further to skolemized ones. In the schematization of these sequents we allow only the matrices of the formulas to contain schematic variables (the number of formulas in the sequents and the quantifier prefixes are fixed).  

\begin{definition}[skolemized prenex sequent schema]\label{def:sps-schema}
Let 
$$S(n) = \Delta_n, \varphi_{1}(n), \cdots, \varphi_{k}(n) \vdash \psi_{1}(n), \cdots,  \psi_{l}(n), \Pi_n, \mbox{ for }k,l\in \mathbb{N} \mbox{ where}$$
\begin{tabular}{ll}
$\varphi_{i}(n) = \forall x_{1}^{i}\cdots \forall x_{\alpha_{i}}^{i} F_{i}(n,x_{1}^{i},\cdots, x_{\alpha_{i}}^{i}),$ \ \ &$\psi_{j}(n) = \exists x_{1}^{j}\cdots \exists y_{\beta_{j}}^{j} E_{j}(n,y_{1}^{j},\cdots, y_{\beta_{j}}^{j}),$
\end{tabular}\\

for $\alpha_{i},\beta_{j}\in \mathbb{N}$, $F_{i}$ and $E_{j}$ are quantifier-free schematic formulas and $\Delta_n,\Pi_n$ are multisets of quantifier-free formulas of fixed size; moreover, the only free variable in any of the formulas is $n:\omega$. Then $S(n)$ is called a skolemized prenex sequent schema (sps-schema).
\end{definition}

\begin{definition}[Herbrand System]
\label{def:herbrand}
Let $S(n)$ be a sps-schema as in Definition~\ref{def:sps-schema}. Then a Herbrand system for $S(n)$ is a rewrite system $\Rcal$ (containing the list constructors and unary function symbols $w_{i}^{x}$, for x $\in \left\lbrace \varphi, \psi \right\rbrace$), such that  for each $\gamma \in \mathbb{N}$, the normal form of $w_{i}^{x}(\gamma)$ w.r.t $\Rcal$ is a list of list of terms $t_{i,x,\gamma}$ (of length $m(i,x)$) such that the sequent 
$$\Delta_\gamma,\Phi_1(\gamma),\ldots,\Phi_k(\gamma) \seq \Psi_1(\gamma),\ldots,\Psi_l(\gamma)$$
for 
\begin{eqnarray*}
\Phi_j(\gamma) &=& \bigwedge^{m(j,\varphi)}_{p=1}E_j(\gamma,t_{j,\varphi,\gamma}(p,1),\ldots,t_{j,\varphi,\gamma}(p,\alpha_j))\ (j=1,\ldots,k),\\
\Psi_j(\gamma) &=& \bigvee^{m(j,\psi)}_{p=1}F_j(\gamma,t_{j,\psi,\gamma}(p,1),\ldots,t_{j,\psi,\gamma}(p,\beta_j))\ (j=1,\ldots,l),
\end{eqnarray*}
is \textbf{LKE}-provable.
\end{definition}

Though our definition of a Herbrand system differs from the definition introduced in~\cite{CERESS2} (where only purely existential schemata are treated), it is only a minor syntactic generalization. All results proven in~\cite{CERESS2} carry over to this more general form above.

\section{``Mathematical'' Proof of the ECA Statement and Discussion of Formal Proof}\label{sec:MathNiA}
For lack of space, we will not provide a formal proof of the ECA-schema in the \textbf{LKS}-calculus (see Appendix \ref{sec:FormECA}), but rather a mathematical argument proving the statement, of which closely follows the intended formal proof. The ECA-schema can be stated as follows: 

\begin{theorem}[Eventually Constant Assertion]
Given a total monotonically decreasing function $f:\mathbb{N}\rightarrow \left\lbrace 0,\cdots , n \right\rbrace $, for $n \in \mathbb{N}$, then there exists an $x \in \mathbb{N}$ such that for all $y \in \mathbb{N}$, where $x\leq y$, it is the case that $f(x) = f(y)$.
\end{theorem}
\begin{proof}
If the range only contains $0$ then the theorem trivially holds. Let us assume it holds for a codomain with $n$ elements and show it hold for a codomain with $n+1$ elements. If for all positions $x$, $f(x)=n$ then the theorem holds, else if at some $y$, $f(y)\not = n$ then from that point on $f$ cannot map to $n$ because the function is monotonically decreasing, thus,  $f$ will only have $n$ elements in its codomain and the theorem holds in this case by the induction hypothesis.
\end{proof}

The cut consists of the case distinction made in the stepcase. When written in the \textbf{LKS}-calculus, it is as follows: 

 $$\exists x \forall y \left( \left( \left( x\leq y  \right) \rightarrow n+1 = f(y) \right) \vee f(y) < n+1 \right)$$

Notice that if we are to formalize the statement  in the \textbf{LKS}-calculus the consequent has a $\exists \forall$ quantifier prefix: 
$$\begin{array}{c}\forall x ( \bigvee_{i=0}^{n+1} i = f(x)) ,  \forall x \forall y \Big(  x\leq y \rightarrow f(y) \leq f(x) \Big)  \vdash   \exists x \forall y ( x\leq y  \rightarrow f(x) = f(y)  ) \end{array}$$
The CERES method (as well as the schematic CERES method) was designed for proofs without {\em strong quantification} in the end sequent. To get around this problem the proofs have to be {\em skolemized} \cite{Baaz:2383422}.We will not go into details of proof skolemization in this work, but to note, in the formal proof $g(\cdot )$, is the introduced skolem symbol.

\section{Extraction of the Characteristic Term Schema}\label{sec:CCSSE}
Each of the proof schema pairs of the formal proof (see Appendix \ref{sec:FormECA}) have one cut configuration. In the case of $\psi$ it is the empty configuration, and in the case of $\varphi(n)$ it is $$\Omega(n)\equiv \exists x \forall y \left( \left( \left( x\leq y  \right) \rightarrow n+1= f(y) \right)    \vee f(y) < n+1 \right).$$ This holds for the basecases as well as the stepcases. Thus, we have the following clause set terms: 

\begin{subequations}
\label{seq:charclaset}
\begin{equation}
\begin{array}{l} CL_{ECA}(0)\equiv \Theta^{\psi,\emptyset}(0)\equiv   cl^{\varphi,\Omega(0)}(0) \oplus   \left(\left\lbrace\vdash f(\alpha)<0 \right\rbrace  \otimes \left\lbrace \vdash 0=f(\alpha)\right\rbrace  \otimes  \right.  \\ \left. \left\lbrace  0\leq \beta \vdash \right\rbrace \right)  \end{array}
\end{equation}
\begin{equation}
\begin{array}{l} cl^{\varphi,\Omega(0)}(0) \equiv\Theta^{\varphi,\Omega(0)}(0) \equiv  \left\lbrace  f(\alpha)<0\vdash \right\rbrace \oplus\left\lbrace  f(g(\alpha))<0\vdash \right\rbrace  \oplus \left\lbrace  \vdash \alpha\leq\alpha \right\rbrace   \\ \oplus \left\lbrace \vdash \alpha\leq g(\alpha)\right\rbrace  \oplus \left\lbrace  0=f(\alpha), 0=f(g(\alpha)) \vdash \right\rbrace  \end{array}
\end{equation}
\begin{equation}
\begin{array}{l} CL_{ECA}(n+1)\equiv \Theta^{\psi,\emptyset}(n+1)\equiv    cl^{\varphi,\Omega(n+1)}(n+1)\oplus   \left(\left\lbrace \vdash f(\alpha)<n+1 \right\rbrace  \right.   \\  \left.   \otimes \left\lbrace  \vdash n+1=f(\alpha) \right\rbrace \otimes   \left\lbrace 0\leq \beta \vdash \right\rbrace  \right)  \end{array}
\end{equation} 

\begin{equation}
\begin{array}{l}  cl^{\varphi,\Omega(n+1)}(n+1) \equiv \Theta^{\varphi,\Omega(n+1)}(n+1)\equiv  {\scriptstyle     cl^{\varphi,\Omega(n)}(n) \oplus  \left\lbrace n+1=f(\alpha),n+1=f(g(\alpha)) \vdash \right\rbrace  \oplus } \\ {\scriptstyle  \left\lbrace \vdash \alpha\leq\alpha\right\rbrace  \oplus \left\lbrace  \alpha\leq g(\alpha)\right\rbrace  \oplus \left\lbrace  n+1=f(\beta)\vdash n+1=f(\beta)\right\rbrace  \oplus \left\lbrace \alpha\leq\beta \vdash \alpha\leq\beta \right\rbrace \oplus  \left\lbrace f(\beta)<n+1 \vdash f(\beta)<n+1\right\rbrace  \oplus  }  \\ {\scriptstyle  \left\lbrace f(\alpha)<n+1,\alpha\leq\beta \vdash n=f(\beta),f(\beta)<n  \right\rbrace   } \end{array}
\end{equation}
\end{subequations}
In the characteristic clause set schema $CL_{ECA}(n+1)$ presented in Eq.\ref{seq:charclaset}  tautology and subsumption elimination have not been applied. Applying both types of elimination to $CL_{ECA}(n)$ and normalizing the clause set yields the following clause set $C(n)$:

$$\begin{array}{ccc}
C1(x,k) & \equiv &\vdash x(k)  \leq  x(k) \\ 
C2(x,k) &\equiv &\vdash x(k) \leq g(x(k)) \\
C3(x,i,k) &\equiv & i = f(x(k)) , i = f(g(x(k))) \vdash \\ 
C4(x,y,i,k) & \equiv & y(k) \leq x(k) , f(y(k))< i+1 \vdash\\ && f(x(k))< i , i = f(x(k))\\
C4'(x,y,i,k) & \equiv & y(k) \leq x(k+1) , f(y(k))< i+1 \vdash\\ && f(x(k+1))< i , i = f(x(k+1))\\
C5(x,k) & \equiv & f(x(k))< 0\vdash  \\
C6(x,k) &\equiv & f(g(x(k)))< 0\vdash  \\
C7(x,k) & \equiv & 0\leq x(k) \vdash f(x(k))< n , f(x(k)) = n
\end{array}$$
We have introduced clause names, schematic variables, and an additional  $\omega$-variable which will be used in the refutation of Sec.~\ref{sec:refuteset}.

\section{Refutation of the Characteristic Clause set of the ECA-schema} \label{sec:refuteset}
We discovered the resolution refutation schema which we present here with the help of the SPASS theorem prover \cite{SpassProver} in default mode, and with the flags for standard resolution and ordered resolution set. Various other modes of the theorem prover were tested, however, given that we needed to translate the resulting proof into the simple resolution language of \cite{CERESS2}, the chosen modes provided the easiest proofs for translation. After running the theorem prover on five instances of the clause set, we where able to extract an invariant for the  resolution refutation schema. Essentially, the refutation differentiates between the symbols occurring in the codomain of $f$ and not occurring. This is denoted using the function $g$. The excerpt from the SPASS output in Table \ref{tabone} indicates the invariant. However, even though SPASS was able to provide a refutation for each instance, we could not use these refutations directly in the resolution refutation schema being that the SPASS output ignores the structural importance of the $\omega$ sort. Unlike the ordering problem of the NiA-schema \cite{MeCadePaper2015,MyThesis}, this choice made by SPASS was not necessary to the refutation of the ECA-schema and we were able find a suitable refutation. 
\begin{table}
\begin{tabular}{l|l}
310[0:MRR:309.0,306.1] &  $\vdash$ \ $f(\alpha) < 3$\ \\ 
311[0:MRR:10.1,310.0] & $\alpha \leq \beta$ $\vdash$ \ $2 = f(\beta)$ \ \ $f(\beta) < 2$\ \\
312[0:Res:2.0,311.0] &  $\vdash$ \ $2 = f(\alpha)$ \ \ $f(\alpha) < 2$\\
314[0:Res:312.0,6.1] & \ $2 = f(\alpha)$ \ $\vdash$ \ $f(g(\beta)) < 2$\ \\
315[0:Res:314.1,11.1] & \ $2 = f(\alpha)$ \ \ $g(\alpha) \leq \beta$ \ $\vdash$ \ $1 = f(\beta)$ \ \ $f(\beta) < 1$\ \\
316[0:Res:312.0,315.0] & \ $g(\alpha) \leq \beta$ \ $\vdash$ \ $f(\alpha) < 2$\ \ $1 = f(\beta)$ \ \ $f(\beta) < 1$\ \\
317[0:Res:2.0,316.0] &  $\vdash$ \ $f(\alpha) < 2$\ \ $1 = f(g(\alpha))$ \ \ $f(g(\alpha)) < 1$\ \\
318[0:Res:3.0,316.0] &  $\vdash$ \ $f(\alpha) < 2$\ \ $1 = f(g(g(\alpha)))$ \ \ $f(g(g(\alpha))) < 1$\ \\
321[0:Res:318.1,7.1] & \ $1 = f(g(\alpha))$ \ $\vdash$ \ $f(\alpha) < 2$\ \ $f(g(g(\alpha))) < 1$\ \\
322[0:Res:321.2,14.1] & \ $1 = f(g(\alpha))$ \ \ $g(g(\alpha)) \leq \beta$ \ $\vdash$ \ $f(\alpha) < 2$\ \ $0 = f(\beta)$ \\
325[0:Res:317.1,322.0] & \ $g(g(\alpha)) \leq \beta$ \ $\vdash$ \ $f(\alpha) < 2$\ \ $f(g(\alpha)) < 1$\  \ $f(\alpha) < 2$\ \ $0 = f(\beta)$ \\
327[0:Obv:325.1] & \ $g(g(\alpha)) \leq \beta$ \ $\vdash$ \ $f(g(\alpha)) < 1$\  \ $f(\alpha) < 2$\ \ $0 = f(\beta)$ \\
328[0:Res:2.0,327.0] &  $\vdash$ \ $f(g(\alpha)) < 1$\  \ $f(\alpha) < 2$\ \ $0 = f(g(g(\alpha)))$ \\
329[0:Res:3.0,327.0] &  $\vdash$ \ $f(g(\alpha)) < 1$\  \ $f(\alpha) < 2$\ \ $0 = f(g(g(g(\alpha))))$ \\
335[0:Res:329.2,8.1] & \ $0 = f(g(g(\alpha)))$ \ $\vdash$ \ $f(g(\alpha)) < 1$\  \ $f(\alpha) < 2$\\
336[0:MRR:335.0,328.2] &  $\vdash$ \ $f(g(\alpha)) < 1$\  \ $f(\alpha) < 2$\\
337[0:Res:336.0,14.1] & \ $g(\alpha) \leq \beta$ \ $\vdash$ \ $f(\alpha) < 2$\ \ $0 = f(\beta)$ \\
338[0:Res:2.0,337.0] &  $\vdash$ \ $f(\alpha) < 2$\ \ $0 = f(g(\alpha))$ \\
339[0:Res:3.0,337.0] &  $\vdash$ \ $f(\alpha) < 2$\ \ $0 = f(g(g(\alpha)))$ \\
344[0:Res:339.1,8.1] & \ $0 = f(g(\alpha))$ \ $\vdash$ \ $f(\alpha) < 2$\\
345[0:MRR:344.0,338.1] &  $\vdash$ \ $f(\alpha) < 2$\\
\end{tabular}
\caption{Excerpt from SPASS output for the clause set instance $C(5)$ indicating the invariant.}\label{tabone}
\end{table}

Our resolution refutation schema of the ECA-schema is $\mathcal{R} = \left(\varrho_{1},\cdots ,\varrho_{10} \right)$, where we use one clause variable $Y$, two schematic variables, and one $\omega$-variable. Our substitution schema is as follows: 
$$\vartheta =\left\lbrace x(k)\leftarrow \lambda k.(h(k)) ,  y(k)\leftarrow \lambda k.(h(k)) \right\rbrace  $$  
where $h(\cdot )$ is defined as  $h(0) \to 0,\ h(s(k)) \to g(h(k))$. The components are as follows:

\begin{equation*}
\begin{array}{ll} \varrho_{1}(n+1,k,x,y,Y) \Rightarrow & r(\varrho_{2}(n+1,k,x,y,Y); \varrho_{5}(n,k,x,y,Y\circ \\ & ( f(x(k)) < n + 1\vdash) );f(x(k)) < n + 1)\\\\

\varrho_{1}(0,k,x,y,Y) \Rightarrow & r(\varrho_{2}(0,k,x,y,Y);C5(x,k) ;f(x(k)) < 0)\\\\

\varrho_{2}(n+1,k,x,y,Y) \Rightarrow  & r(\varrho_{3}(n+1,k,x,y,Y);r(C1(x,k);\\ 
& C7(x,k);x(k)\leq x(k));n+1 = f(x(k)))\\\\

\varrho_{2}(0,k,x,y,Y) \Rightarrow &  r(\varrho_{3}(0,k,x,y,Y);r(C1(x,k);C7(x,k);x(k)\leq x(k));\\ & n+1 = f(x(k)))\\\\
\end{array}
\end{equation*}
\begin{equation*}
\begin{array}{ll}
\varrho_{3}(n+1,k,x,y,Y) \Rightarrow & r(\varrho_{4}(n+1,k,x,y,Y);C3(x,n+1,k); \\ & n+1=f(g(x(k))))\\\\

\varrho_{3}(0,k,x,y,Y) \Rightarrow & r(\varrho_{4}(0,k,x,y,Y);C3(x,0,k);0=f(g(x(k))))\\\\

\varrho_{4}(n+1,k,x,y,Y) \Rightarrow & r(\varrho_{5}(n,k+1,x,y,Y\circ f(x(k+1))<n+1\vdash );\\ & r(C2(x,k);C7(x,k+1);f(x(k+1))<n+1) \\\\

\varrho_{4}(0,k,x,y,Y) \Rightarrow &  r(C6(x,k);r(C2(x,k);C7(x,k+1);f(g(x(k)))<0)
\end{array}
\end{equation*}
\begin{equation*}
\begin{array}{ll}
\varrho_{5}(n+1,k,x,y,Y) \Rightarrow &  r(\varrho_{6}(n+1,k,x,y,Y);\varrho_{5}(n,k,x,y,Y\circ \\ &  ( f(x(k)) < n+1 \vdash)) ;f(x(k)) < n + 1)\\\\

\varrho_{5}(0,k,x,y,Y) \Rightarrow & r(\varrho_{6}(0,k,x,y,Y);C5(x,k) ;f(x(k)) < 0)\\\\

\varrho_{6}(n+1,k,x,y,Y) \Rightarrow &  r(\varrho_{7}(n+1,k,x,y,Y);\varrho_{8}(n+1,k,x,y,Y)\\ & ;n+1 = f(x(k)))\\\\
\varrho_{6}(0,k,x,y,Y) \Rightarrow & r(\varrho_{7}(0,k,x,y,Y);\varrho_{8}(0,k,x,y,Y); 0 = f(x(k))\\\\
\end{array}
\end{equation*}
\begin{equation*}
\begin{array}{ll}
\varrho_{7}(n+1,k,x,y,Y) \Rightarrow & r(\varrho_{9}(n+1,k,x,y,Y);C3(x,n+1,k);\\ &n+1=f(g(x(k))))\\\\

\varrho_{7}(0,k,x,y,Y) \Rightarrow & r(\varrho_{9}(0,k,x,y,Y);C3(x,0,k);0=f(g(x(k))))\\\\

\varrho_{8}(n+1,k,x,Y) \Rightarrow & r(C1(x,k);Y\circ C4(x,y,n,k);x(k)\leq x(k))\\\\

\varrho_{8}(0,k,x,y,Y)\Rightarrow & r(C1(x,k);Y\circ C4(x,y,0,k);x(k)\leq x(k))\\\\
\end{array}
\end{equation*}
\begin{equation*}
\begin{array}{ll}
\varrho_{9}(n+1,k,x,y,Y) \Rightarrow & r(\varrho_{5}(n,k+1,x,y,Y'\circ f(g(x(k)))<n+1\vdash );\\ & \varrho_{10}(n+1,x,y,Y);f(g(x(k)))<n+1)\\\\

\varrho_{9}(0,k,x,y,Y) \Rightarrow &  r( C6(x,k) ;\varrho_{10}(0,k,x,y,Y) ;f(g(x(k)))<0)\\\\

\varrho_{10}(n+1,k,x,y,Y) \Rightarrow & r(C2(x,k);Y\circ C4'(x,y,n,k);x(k)\leq g(x(k)))\\\\

\varrho_{10}(0,k,x,y,Y) \Rightarrow & r(C2(x,k);Y\circ C4'(x,y,0,k);x(k)\leq g(x(k)))\end{array}
\end{equation*}

\begin{figure}
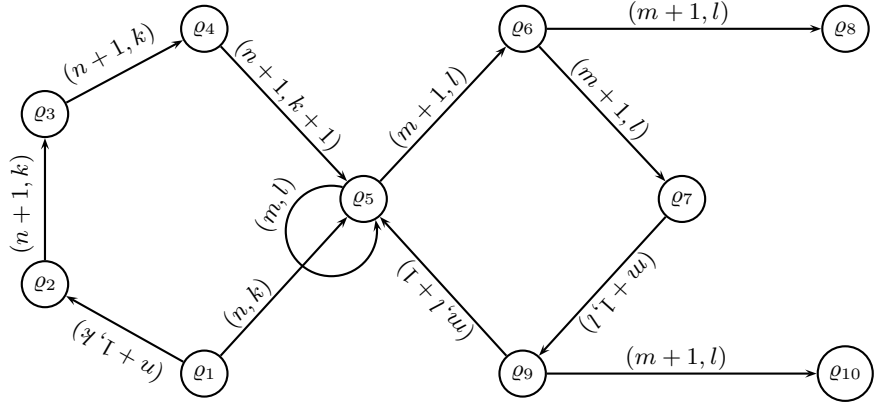

\begin{center}
$
\psmatrix[colsep=1.45cm,rowsep=.45cm,mnode=circle]
 &\varrho_{4}&&\varrho_{6}&&\varrho_{8}\\
\varrho_{3}\\
 &&\varrho_{5}&&\varrho_{7}\\
\varrho_{2}\\
 &\varrho_{1}&&\varrho_{9}&&\varrho_{10}
\ncline{->}{1,2}{3,3}
\lput{:U}{\rput[r]{0}(.8 ,.2){(n+1,k+1)}}
\ncline{->}{5,2}{3,3}
\lput{:U}{\rput[r]{0}(-.1,.2){(n,k)}}
\ncline{->}{5,2}{4,1}
\lput{:U}{\rput[r]{0}(.6,.3){(n+1,k)}}
\ncline{->}{4,1}{2,1}
\lput{:U}{\rput[r]{0}(.6,.3){(n+1,k)}}
\ncline{->}{2,1}{1,2}
\lput{:U}{\rput[r]{0}(.6,.3){(n+1,k)}}
\ncline{->}{3,3}{1,4}
\lput{:U}{\rput[r]{0}(.6,.2){(m+1,l)}}
\ncline{->}{1,4}{3,5}
\lput{:U}{\rput[r]{0}(.6,.2){(m+1,l)}}
\ncline{->}{3,5}{5,4}
\lput{:U}{\rput[r]{0}(.6,.2){(m+1,l)}}
\ncline{->}{5,4}{3,3}
\lput{:U}{\rput[r]{0}(.6,.2){(m,l+1)}}
\ncline{->}{5,4}{5,6}
\lput{:U}{\rput[r]{0}(.6,.2){(m+1,l)}}
\ncline{->}{1,4}{1,6}
\lput{:U}{\rput[r]{0}(.6,.2){(m+1,l)}}
\nccircle[nodesep=0pt,angleA=135]{->}{3,3}{.6cm}
\lput{:U}{\rput[r]{110}(-.9,.7){(m,l)}}
\endpsmatrix
$
\end{center}
\caption{A graph representation of the resolution refutation. The variable $n$ is the free parameter, $k$ is the $\omega$-variable used in the refutation and the variables $m$ and $l$ are dependent on the position in the computation.}
\label{fig:spaceship}
\end{figure}

One can find a graphical representation of the refutation in Fig.~\ref{fig:spaceship}. The clause substitution is $\theta = \lbrace Y\leftarrow\ \vdash \rbrace$, the $\omega$-variable substitution is $\nu = \lbrace k\leftarrow \overline{\mu} \rbrace$ for any $\overline{\mu}\in \mathbb{N}$.
The normal form of the refutation for $\gamma\in \mathbb{N}$ is  
$$\varrho_1 (n,k,x,y, Y)\theta\nu\vartheta\left[n\leftarrow \gamma \right] = \varrho_1 (\gamma,\overline{\mu},\lambda_k.(i_s(k)),\lambda_k.(h(k)), \vdash ),$$ 
where $i_s(0) =0,\ i_s(s(k)) = s(i_s(k))$. Substitution of the empty clause into $Y$ suffices for every instance, i.e.  $\lbrace Y\leftarrow\ \vdash \rbrace$. This property makes extraction of the Herbrand system much easier.

\section{The Herbrand System for the ECA-schema} \label{sec:Herbrand}

Now we move on to the construction of a  Herbrand system for the sequent

$$\begin{array}{l}S(n) \equiv \left( \forall x \bigvee_{i=0}^{n} i = f(x),   \forall y \left( 0\leq y \rightarrow f(y) \leq f(0)\right)  \right)  \vdash \\  \exists x \left( x \leq g(x)  \rightarrow f(x) = f(g(x)) \right) \end{array}$$

based on our proof analysis. The sequent $S(n)$ is an sps-schema of the form $\varphi_1(n),\varphi_2(n)\vdash \psi_1(n)$. Note that we dropped one of the quantifiers from the antecedent being that it is obvious from the proof itself what the substitution would be, see Appendix ~\ref{sec:FormECA}. Each formula in $S(n)$ is derived along with a set of clauses in the proof schemata $\Psi =\left\langle (\psi(n+1),\psi(0)), (\varphi(n+1),\varphi(0)) \right\rangle$. By observing the construction of the 
formulae in Appendix ~\ref{sec:FormECA}.\ref{Fig:proofPSIone} and  Appendix ~\ref{sec:FormECA}.\ref{Fig:proofPHIone},  one can see that $\varphi_1(n)$, $\varphi_2(n)$, and $C7(x,k)$ as constructed together, while $\psi_1(n)$, $C2(x,k)$, and $C3(x,i,k)$ are constructed together. We will only consider the case when  the $\omega$-variable substitution is $\nu = \lbrace k \rightarrow 0\rbrace$ to simplify the derivation. 

Notice that $C7(x,k)$ is used at the top of the refutation and only twice. Once as $C7(x,0)$ and once as $C7(x,1)$. On the other hand, $C2(x,k)$ is used in $\varrho_{10}$ and $C3(x,i,k)$ is use in $\varrho_{7}$. For every pair $(i,l)$ in the ranges $0\leq i\leq n+1$,$0\leq l< n+1$, the clauses $C2(x,l)$ and  $C3(x,i,l)$, and $C2(x,l+1)$ and  $C3(x,i,l+1)$ are used in the refutation. This implies, by the substitution schema that $\psi_1(n)$ will have its quantifier replaced by the term derived from $h(i)$, for all  $i\in \left[ 0,n\right] $, in the Herbrand system.  This information can be used to construct the required rewrite system: 

$$ \mathcal{R} = \left\lbrace  \begin{array}{l} w_{1}^{\varphi}(k+1) \Rightarrow [[0];[g(0)]] \\ w_{1}^{\varphi}(0) \Rightarrow [[0];[g(0)]]\\\\
 w_{2}^{\varphi}(k+1) \Rightarrow [[0];[g(0)]] \\ w_{2}^{\varphi}(0) \Rightarrow [[0];[g(0)]]  \\\\
 w_{1}^{\psi}(k+1) \Rightarrow [[h(k+1)];w_{1}^{\psi}(k)] \\ w_{1}^{\psi}(0) \Rightarrow [0]\end{array}\right. $$
To finish our construction of the Herbrand system using Def. \ref{def:herbrand} We need to put all of the parts together as a single sequent as follows

$$ \bigvee_{i=0}^{n} i = f(0),\bigvee_{i=0}^{n} i = f(g(0)), \left( 0\leq 0 \rightarrow f(0) \leq f(0)\right), $$
$$  \left( 0\leq g(0) \rightarrow f(g(0)) \leq f(0)\right) \vdash  \bigvee_{i=0}^{n} \left( h(i) \leq g(h(i))  \rightarrow f(h(i)) = f(g(h(i))) \right) . $$

At first this does not seem to be \textbf{LKE} provable, However, one has to remember that for the construction of our cut formula we made an assumption that $f$ is monotonically decreasing and has a codomain consisting of elements in the interval $\left[0,n \right]$. These assumptions are represented by the following axiom found in  Appendix ~\ref{sec:FormECA}:
$$AX\equiv  f(\alpha) < n+1 ,  \alpha\leq \beta  \vdash    n= f(\beta)  , f(\beta) < n $$

It is not used in the construction of the end sequent but is used for the construction of the cut formulae. We just need to find a set of axioms which correspond to these semantic assumptions, the following set suffices:

\begin{center}
\begin{tabular}{ll}
\multicolumn{2}{l}{$\begin{array}{c}A1(i):\  \bigvee_{i=0}^{j-1} i= f(\alpha), j =f(g(\alpha)), f(g(\alpha)) < f(\alpha)  \vdash \end{array}$}\\
\multicolumn{2}{l}{$\begin{array}{cc}A2(i):\ i= f(\alpha), \bigvee_{j=0}^{i-1}  j=f(g(\alpha)),   \alpha\leq g(\alpha)  \vdash  \end{array}$}\\
\multicolumn{2}{l}{$\begin{array}{c}A3(i):\ i= f(\alpha), i =f(g(\alpha)) \vdash f(\alpha) = f(g(\alpha))  \end{array}$} \\ 
\multicolumn{2}{l}{$\begin{array}{c}A4(i):f(g(\alpha)) = f(\alpha)  \vdash f(\alpha) = f(g(\alpha))  \end{array}$}\\
$\begin{array}{c}A5(i):\  \vdash \alpha \leq \alpha  \end{array}$ & \\ $\begin{array}{c}A6(i): f(\alpha) < f(\alpha)   \vdash \end{array}$
\end{tabular}
\end{center}

The first pair of axioms enforce the required properties of $f$ and $g$, the next pair provide the needed properties of equality, and the last pair provide the needed properties of linear orderings. Interesting enough, using these axioms, we are able to prove the derived Herbrand sequent using only a single nesting of $g$, thus making the majority of the consequent redundant.  This is a result of our usage of the clause $C7(x,k)$. Thus, it turns out that a minimal Herbrand sequent is the following: 
$$ \bigvee_{i=0}^{n} i = f(0),\bigvee_{i=0}^{n} i = f(g(0)),  \left( 0\leq g(0) \rightarrow f(g(0)) \leq f(0)\right),  $$ $$ \left( 0\leq 0 \rightarrow f(0) \leq f(0)\right) \vdash  0 \leq g^{1}(0)  \rightarrow f(0) = f(g^{1}(0))   . $$
The Herbrand sequent can be derived for deeper nestings of $g$ by changing the $\omega$-variable substitution used.

\section{Conclusion} \label{sec:Conclusion}

Weakening the NiA-schema of \cite{MeCadePaper2015} by reducing the complexity of the cuts  allowed for extraction of the Herbrand system using the concepts of \cite{CERESS2}. As a case study of the schematic CERES method, to the best of our knowledge this is the first one. From the analysis of the ECA-schema there are two issues which seem to influence the applicability  of the  schematic CERES method. The first issue, as we pointed out earlier, is the ordering of the terms in the $\omega$ sort. However, a second issue arising in this work is the complexity of the terms,  specifically what is the highest arity function symbol allowed. In the case of the NiA-schema, terms were constructed from both an arity two  and an arity one function symbol, but in the case of the ECA-schema only  arity one function symbols where used. When only arity one function symbols are used nesting of the function symbols does not require the addition of {\em extra variables} in a given term, of which were used in the NiA-schema \cite{MeCadePaper2015,MyThesis}. This seems to allow for the creation of more complex orderings of the $\omega$ sort. We conjecture a sufficient condition that proof schema containing only arity one function symbols can be analysed using the schematic CERES method. Also, an open problem we plan to address in future work is a generalization of the resolution refutation calculus of \cite{CERESS2} which can handle more complex ordering structures\cite{MyThesis}. It seems necessary to handle more complex ordering structure if one wants to formalize and analyse more complex mathematical arguments such as F\"{u}rstenberg's proof of the infinitude of primes.

\bibliography{references}

\newpage
\appendix

\section{ECA Formalized in the \textbf{LKS}-calculus}\label{sec:FormECA}

In our \textbf{LKS}-calculus, cut ancestors have a $^{*}$ and cut-configuration ancestors have a $^{**}$. The proof has already been skolemized. We will make the following abbreviations to simplify the formal proof:\\ $ESC \equiv \exists x ( x\leq g(x)  \rightarrow f(x) = f(g(x))  ) $,\\ $MD \equiv \forall x \forall y \Big(  x\leq y \rightarrow f(y) \leq f(x) \Big)$, \\ $FD(n)\equiv \forall x ( \bigvee_{i=0}^{n} i = f(x))$ ,\\ $CUT \equiv \exists x \forall y \left(  \left( x\leq y  \right) \rightarrow n+1= f(y) \vee f(y) < n+1 \right)$,\\ $ AX\equiv  f(\alpha) < n+1^{*} ,  \alpha\leq \beta^{*}  \vdash    n= f(\beta)^{*}  , f(\beta) < n ^{*}$.\\ Also, We will remove every inference rule which does not influence the characteristic clause set of the ECA-schema.

\begin{figure}[H]
\begin{tiny}
\begin{prooftree}
\AxiomC{$\begin{array}{c} \vdots\\
\bigvee_{i=0}^{n} i = f(\beta)    \vdash \\    f(\beta) < n+1^{*}  \end{array}$}
 \UnaryInfC{$\begin{array}{c} 
\vdots \end{array}$}
 \AxiomC{$\begin{array}{c} 
 n+1= f(\beta)  \vdash \\  n+1= f(\beta)^{*} \end{array}$}
 \RightLabel{$\vee:l$}
 \BinaryInfC{$\begin{array}{c} 
\vdots \end{array}$}
 \AxiomC{$\begin{array}{c} 
 0 \leq \beta^*  \vdash \\  0 \leq \beta \end{array}$}
 \RightLabel{$\rightarrow:l$}
 \BinaryInfC{$\vdots$}
  \UnaryInfC{$\begin{array}{c} 
FD(n+1),MD  \vdash  CUT(n+1)^{*} \end{array}$}

 \AxiomC{$\begin{array}{c}  \varphi(n+1) \end{array}$}
\dottedLine
 \UnaryInfC{$\begin{array}{c}  CUT(n+1)^{*} \vdash ESC \end{array}$}
 \RightLabel{$cut$}
\BinaryInfC{$\begin{array}{c}FD(n+1),MD\vdash ESC \end{array}$}
\end{prooftree}
\end{tiny}
\caption{Proof symbol $\psi(n+1)$}
\label{Fig:proofPSIone}
\end{figure}

\begin{figure}[H]
\begin{tiny}
\begin{prooftree} 
 \AxiomC{$\begin{array}{c} 
 0= f(\alpha)  \vdash   0= f(\alpha)^{*} \end{array}$}
 \UnaryInfC{$\begin{array}{c}  \vdots \end{array}$}
 \AxiomC{$\begin{array}{c} 
 0 \leq \alpha^* \vdash  0 \leq \alpha \end{array}$}
 \RightLabel{$\rightarrow:l$}
 \BinaryInfC{$\begin{array}{c} \vdots\end{array}$}
 \UnaryInfC{$\begin{array}{c} 
FD(0) , MD \vdash CUT(0)^{*} \end{array}$}
\AxiomC{$\begin{array}{c}  \varphi(0) \end{array}$}
\dottedLine
 \UnaryInfC{$\begin{array}{c}  CUT(0)^{*} \vdash ESC \end{array}$}
 \RightLabel{$cut$}
\BinaryInfC{$\begin{array}{c}FD(0) , MD  \vdash ESC \end{array}$}
\end{prooftree}
\end{tiny}
\caption{Proof symbol $\psi(0)$}
\end{figure}

\begin{figure}[H]
\begin{tiny}
\begin{prooftree}
\AxiomC{$\begin{array}{c}   n+1= f(\beta)^{**} \vdash \\      n+1 = f(\beta) ^{*}\end{array}$}
\AxiomC{$\begin{array}{c} \alpha\leq \beta^{*} \vdash \\      \alpha\leq \beta^{**}\end{array}$}
  \RightLabel{$\rightarrow:l$}

\BinaryInfC{$\begin{array}{c}  \vdots \end{array}$}
 \AxiomC{$\begin{array}{c}   f(\beta) < n+1 ^{**} \vdash \\ f(\beta) < n+1^{*} \end{array}$}
  \RightLabel{$\vee:l$}
\BinaryInfC{$\begin{array}{c} \vdots \end{array}$}

\AxiomC{$\begin{array}{c}AX\end{array}$}
 \RightLabel{$\rightarrow:r$}
 \RightLabel{$cut$}
 \BinaryInfC{$\begin{array}{c} (1)\end{array}$}
 \end{prooftree}
\end{tiny}

\begin{tiny}
\begin{prooftree}
\AxiomC{$\begin{array}{c}  (1) \end{array}$}
\AxiomC{$\begin{array}{c}   n+1 = f(\alpha) ^{*} ,\\  n+1 = f(g(\alpha)) ^{*} \vdash \\     f(\alpha) = f(g(\alpha))  \end{array}$}
\AxiomC{$\begin{array}{c}   \vdash \alpha \leq \alpha^{*}\end{array}$}

 \RightLabel{$\rightarrow:l$}
\BinaryInfC{$\begin{array}{c} \vdots  \end{array}$}
\AxiomC{$\begin{array}{c}   \alpha \leq g(\alpha)\vdash \\ \alpha \leq g(\alpha)^{*}\end{array}$}
 \RightLabel{$\rightarrow:l$}
\BinaryInfC{$\begin{array}{c} \vdots \end{array}$}
 \RightLabel{$cut$}
 \BinaryInfC{$\begin{array}{c}  CUT(n+1)^{**} \vdash  CUT(n)^{*}, ESC  \end{array}$}

\AxiomC{$\begin{array}{c}  \varphi(n) \end{array}$}
\dottedLine
 \UnaryInfC{$\begin{array}{c} CUT(n)^{*} \vdash ESC\end{array}$}
 \RightLabel{$cut$}
 \BinaryInfC{$\begin{array}{c} CUT(n+1)^{**} \vdash ESC,ESC \end{array}$}
  \RightLabel{$c:l$}
 \UnaryInfC{$\begin{array}{c}   CUT(n+1)^{**} \vdash ESC \end{array}$}
 \end{prooftree}
\end{tiny}
\caption{Proof symbol $\varphi(n+1)$}
\label{Fig:proofPHIone}
\end{figure}

\begin{figure}[H]
\begin{tiny}
\begin{prooftree}

\AxiomC{$\begin{array}{c}   0= f(\alpha)^{**}, \\ 0= f(g(\alpha)) ^{**}  \vdash \\    f(\alpha) = f(g(\alpha))  \end{array}$}
\AxiomC{$\begin{array}{c}  \alpha \leq g(\alpha)\vdash  \\ \alpha\leq g(\alpha) ^{**}  \end{array}$}
\RightLabel{$\rightarrow:l$}
\BinaryInfC{$\begin{array}{c}   \vdots \end{array}$}
\AxiomC{$\begin{array}{c}  \vdash  \alpha\leq \alpha  ^{**}  \end{array}$}
\RightLabel{$\rightarrow:l$}
\BinaryInfC{$\begin{array}{c} \vdots   \end{array}$}

\AxiomC{$\begin{array}{c}    f(g(\alpha)) < 0 ^{**} \vdash  \end{array}$}
\RightLabel{$\vee:l$}
\BinaryInfC{$\begin{array}{c}  \vdots \end{array}$}
\end{prooftree}
\end{tiny}
\begin{tiny}
\begin{prooftree}
\AxiomC{$\begin{array}{c}  \vdots \end{array}$}
\AxiomC{$\begin{array}{c}    f(\alpha) < 0 ^{**} \vdash  \end{array}$}
 \RightLabel{$\vee:l$}
\BinaryInfC{$\begin{array}{c}  \vdots \end{array}$}
 \UnaryInfC{$\begin{array}{c} CUT(0)^{**} \vdash  ESC \end{array}$}
 \end{prooftree}
\end{tiny}
\caption{Proof symbol $\varphi(0)$}
\end{figure}

\section{Example Resolution Refutation}

\begin{figure}
\begin{tiny}
\begin{prooftree}
\AxiomC{$ C6(x,0,2)$} 
\AxiomC{$C2(x,0,2)$} 
\AxiomC{$C4'(x,y,0,2)$} 
\BinaryInfC{$f(g(g(0))) < 1 \vdash 0= f(g(g(g(0)))),  f(g(g(g(0))))<0$} 
\BinaryInfC{$f(g(g(0))) < 1 \vdash 0= f(g(g(g(0))))$} 
\AxiomC{$C3(x,0,2)$} 
\BinaryInfC{$f(g(g(0))) < 1,0=f(g(g(0)))\vdash(6)$} 
\end{prooftree}
\end{tiny}
\begin{tiny}
\begin{prooftree}
\AxiomC{$\begin{array}{c} f(g(g(0))) < 1,\\ 0=f(g(g(0)))\vdash (6)\end{array}$} 
\AxiomC{$C1(x,0,2)$} 
\AxiomC{$C4(x,0,2)$} 
\BinaryInfC{$\begin{array}{c} f(g(g(0))) < 1\vdash\\ 0= f(g(g(0))),  f(g(g(0)))<0 \end{array}$} 
\BinaryInfC{$f(g(g(0))) < 1\vdash  f(g(g(0)))<0$} 
\AxiomC{$ C5(x,0,2)  $} 
\BinaryInfC{$f(g(g(0))) < 1\vdash (D)$}
\end{prooftree}
\end{tiny}
\caption{Resolution refutation for instance $n=2$ and $k=0$ (Part B).}
\end{figure}
\begin{figure}
\begin{tiny}
\begin{prooftree}
\AxiomC{$ C6(x,0,1)$} 
\AxiomC{$C2(x,0,1)$} 
\AxiomC{$C4'(x,y,0,1)$} 
\BinaryInfC{$f(g(0)) < 1 \vdash 0= f(g(g(0))),  f(g(g(0)))<0$} 
\BinaryInfC{$f(g(0)) < 1 \vdash 0= f(g(g(0)))$} 
\AxiomC{$C3(x,0,1)$} 
\BinaryInfC{$f(g(0)) < 1,0=f(g(0))\vdash(5)$} 
\end{prooftree}
\end{tiny}
\begin{tiny}
\begin{prooftree}
\AxiomC{$f(g(0)) < 1,0=f(g(0))\vdash (5)$} 
\AxiomC{$ C1(x,0,1)$} 
\AxiomC{$C4(x,y,0,1)$} 
\BinaryInfC{$f(g(0)) < 1\vdash 0= f(g(0)),  f(g(0))<0 $} 
\BinaryInfC{$f(g(0)) < 1\vdash  f(g(0))<0$} 
\AxiomC{$ C5(x,0,1) $} 
\BinaryInfC{$f(g(0)) < 1\vdash (C)$}
\end{prooftree}
\end{tiny}
\begin{tiny}
\begin{prooftree}
\AxiomC{$ C6(x,0,0) $} 
\AxiomC{$C2(x,0,0)$} 
\AxiomC{$C4'(x,y,0,0)$} 
\BinaryInfC{$f(0) < 1 \vdash 0= f(g(0)),  f(g(0))<0$} 
\BinaryInfC{$f(0) < 1 \vdash 0= f(g(0))$} 
\AxiomC{$C3(x,0,0)$} 
\BinaryInfC{$f(x(1,0)) < 1,0=f(x(0,0))\vdash(4)$} 
\end{prooftree}
\end{tiny}
\begin{tiny}
\begin{prooftree}
\AxiomC{$f(0) < 1,0=f(0)\vdash (4)$} 
\AxiomC{$C1(x,0,0)$} 
\AxiomC{$C4(x,y,0,0)$} 
\BinaryInfC{$f(0) < 1\vdash 0 = f(0),  f(0)<0 $} 
\BinaryInfC{$f(0) < 1\vdash  f(0)<0$} 
\AxiomC{$C5(x,0,0) $} 
\BinaryInfC{$f(0) < 1\vdash (E)$}
\end{prooftree}
\end{tiny}
\begin{tiny}
\begin{prooftree}
\AxiomC{$ f(g(g(0)))<1\vdash (D) $} 
\AxiomC{$C2(x,1,1)$} 
\AxiomC{$C4'(x,y,1,1)$} 
\BinaryInfC{$\begin{array}{c} f(g(0)) < 2\vdash\\ 1 = f(g(g(0))),  f(g(g(0)))<1 \end{array}$} 
\BinaryInfC{$f(g(0)) < 2\vdash 1 = f(g(g(0)))$} 
\AxiomC{$C3(x,1,1) $} 
\BinaryInfC{$f(g(0)) < 2,1 = f(g(0))\vdash (3) $} 
\end{prooftree}
\end{tiny}
\begin{tiny}
\begin{prooftree}
\AxiomC{$f(g(0)) < 2, 1 = f(g(0))\vdash (3) $} 
\AxiomC{$C1(x,1,1)$} 
\AxiomC{$C4(x,y,1,1)$} 
\BinaryInfC{$\begin{array}{c} f(g(0)) < 2 \vdash \\ 1 = f(g(0)),  f(g(0))<1 \end{array}$} 
\BinaryInfC{$ f(g(0)) < 2 \vdash  f(g(0))<1 $} 
\AxiomC{$ f(g(0)) < 1 \vdash (C)$} 
\BinaryInfC{$f(g(0)) < 2\vdash (B)$}
\end{prooftree}
\end{tiny}
\begin{tiny}
\begin{prooftree}
\AxiomC{$f(g(0))<1\vdash (C)$} 
\AxiomC{$C2(x,1,0)$} 
\AxiomC{$C4'(x,y,1,0)$} 
\BinaryInfC{$f(0) < 2\vdash 1 = f(g(0)),  f(g(0))<1$} 
\BinaryInfC{$f(0) < 2\vdash 1 = f(g(0))$} 
\AxiomC{$\begin{array}{c} C3(x,1,0)\vdash \end{array} $} 
\BinaryInfC{$f(0) < 2,1 = f(0)\vdash (2) $} 
\end{prooftree}
\end{tiny}
\begin{tiny}
\begin{prooftree}
\AxiomC{$f(0) < 2, 1 = f(0)\vdash (2) $} 
\AxiomC{$C1(x,1,0) $} 
\AxiomC{$ C4(x,y,1,0) $} 
\BinaryInfC{$ f(0) < 2 \vdash 1 = f(0),  f(0)<1 $} 
\BinaryInfC{$ f(0) < 2 \vdash  f(0)<1 $} 
\AxiomC{$ f(0) < 1 \vdash (E)$} 
\BinaryInfC{$f(0) < 2\vdash (A)$}
\end{prooftree}
\end{tiny}
\begin{tiny}
\begin{prooftree}
\AxiomC{$f(g(0))<2\vdash (B)$} 
\AxiomC{$C2(x,2,0)$} 
\AxiomC{$C7(x,2,1)$} 
\BinaryInfC{$\vdash 2 = f(g(0)),  f(g(0))<2$} 
\BinaryInfC{$\vdash 2 = f(g(0))$} 
\AxiomC{$C3(x,2,0)$} 
\BinaryInfC{$2=f(0)\vdash (1) $} 
\end{prooftree}
\end{tiny}
\begin{tiny}
\begin{prooftree}
\AxiomC{$2=f(0)\vdash (1) $} 
\AxiomC{$\vdash C1(x,2,0) $} 
\AxiomC{$C7(x,2,0)$} 
\BinaryInfC{$\vdash 2 = f(0),  f(0)<2$} 
\BinaryInfC{$\vdash f(0)<2 $} 
\AxiomC{$ f(0) < 2\vdash (A)$} 
\BinaryInfC{$\vdash $}
\end{prooftree}
\end{tiny}
\caption{Resolution refutation for instance $n=2$ and $k=0$ (Part A).}
\end{figure}
\end{document}